\documentclass{aims}
\usepackage{amsmath}
\usepackage{paralist}
\usepackage{graphics} 
\usepackage{epsfig} 
\usepackage{graphicx}  \usepackage{epstopdf}
\usepackage[dvipsnames]{xcolor}
\usepackage[colorlinks=true]{hyperref}
\hypersetup{urlcolor=blue, citecolor=red}

\usepackage{booktabs}
\usepackage{empheq}
\usepackage{tikz-cd}

\usepackage{bm}
\usepackage{mathtools}
\usepackage{cleveref}
\usepackage{xspace}
\usepackage{csquotes}

\usepackage{palatino}

\newtheorem{theorem}{Theorem}[section]
\newtheorem{corollary}{Corollary}[section]

\newtheorem{lemma}[theorem]{Lemma}
\newtheorem{proposition}{Proposition}

\theoremstyle{definition}
\newtheorem{definition}[theorem]{Definition}
\newtheorem{remark}{Remark}[section]

\crefname{theorem}{Theorem}{Theorems}
\crefname{corollary}{Corollary}{Corollaries}
\crefname{hypothesis}{Hypothesis}{Hypotheses}
\crefname{notation}{Notation}{Notations}
\crefname{definition}{Definition}{Definitions}
\crefname{remark}{Remark}{Remarks}
\crefname{example}{Example}{Examples}
\crefname{lemma}{Lemma}{Lemmata}
\crefname{proposition}{Proposition}{Propositions}
\crefname{assumption}{Assumption}{Assumptions}
\crefname{figure}{Figure}{Figures}

\crefformat{equation}{(#2#1#3)}
\crefrangeformat{equation}{(#3#1#4) to~(#5#2#6)}
\crefmultiformat{equation}{(#2#1#3)}{ and~(#2#1#3)}{, (#2#1#3)}{ and~(#2#1#3)}

\newcommand{\Cal}[1]{\mathcal{#1}}
\newcommand{\Op}[1]{\mathrm{#1}}
\newcommand{\RR}{\mathbb{R}}
\newcommand{\Rny}{{\mathbb{R}^{n_y}}}
\newcommand{\Rnx}{{\mathbb{R}^{n_x}}}
\newcommand{\Bf}[1]{\boldsymbol{#1}}
\newcommand{\BfX}{{\Bf{X}}}
\newcommand{\BfCalX}{\boldsymbol{\Cal X}}

\newcommand{\BfCalXN}{\boldsymbol{\Cal X}^{(N)}}

\newcommand{\BfXi}{{\Bf{X}_i}}
\newcommand{\BfXj}{{\Bf{X}_j}}

\newcommand{\Bfx}{{\Bf{x}}}
\newcommand{\Bfy}{{\Bf{y}}}
\newcommand{\Bfv}{{\Bf{v}}}
\newcommand{\init}[1]{{#1}^{\mathrm{in}}}
\newcommand{\Dif}{\mathrm{\partial}}
\newcommand{\DifX}{\Dif_\BfX}
\newcommand{\DifXi}{\Dif_{\BfX_i}}
\newcommand{\DifXj}{\Dif_{\BfX_j}}
\newcommand{\Dify}{\Dif_\Bfy}
\newcommand{\PRnx}[1]{\mathcal{P}^{#1}(\Rnx)}
\newcommand{\fai}{\quad \text{for~} 1 \leq i \leq N}

\newcommand{\dif}{\mathrm{d}}

\newcommand{\Inv}{\mathrm{Inv}}	

\newcommand{\norm}[1]{\left\Vert #1 \right\Vert}

\newcommand{\eff}[1]{{#1}_\mathrm{eff}}   
\newcommand{\mf}[1]{{#1}_\mathrm{mf}}     
\newcommand{\bN}{{(N)}}  

\newcommand{\effN}[1]{{#1}_\mathrm{eff}^{\bN}}   

\newcommand{\effO}[1]{\eff{#1}^{(1)}}
\newcommand{\mfO}[1]{\mf{#1}^{(1)}}

\newcommand{\pN}{{(N)}}

\newcommand{\simeon}[1]{\textcolor{Black}{#1}}

\newcommand{\crefDAE}{\cref{eq:x_micro_ind3,eq:y_micro_ind3,eq:g_micro_ind3}\xspace}
\newcommand{\crefODE}{\cref{eq:y_micro_ode,eq:x_micro_ode}\xspace}

\title[The mean-field limit for particle systems with uniform full-rank constraints] 
{The mean-field limit for particle systems with uniform full-rank constraints}

\author[Steffen Plunder and Bernd Simeon]{}


\email{steffen.plunder@univie.ac.at}
\email{simeon@mathematik.uni-kl.de}

\thanks{$^*$ Corresponding author: Steffen Plunder}

\begin{document}
	\maketitle
	
	\centerline{\scshape Steffen Plunder$^*$}
	\medskip
	{\footnotesize
		\centerline{University of Vienna}
		\centerline{Department of Mathematics}
		\centerline{Oskar-Morgenstern-Straße 4, 1090 Vienna, Austria}
	} 
	
	\medskip
	
	\centerline{\scshape Bernd Simeon}
	\medskip
	{\footnotesize
		\centerline{Technische Universität Kaiserslautern}
		\centerline{Felix-Klein Zentrum}
		\centerline{Paul-Ehrlich Straße 31, 67663 Kaiserslautern}
	}
	
	\bigskip
	

	\begin{abstract}
		We consider a particle system with uniform coupling between a macroscopic component and individual particles. The constraint for each particle is of full rank, which implies that each movement of the macroscopic component leads to a movement of all particles and vice versa. 
		Skeletal muscle tissues share a similar property which motivates this work.		
		
		We prove convergence of the mean-field limit, well-posedness and a stability estimate for the mean-field PDE.		
		This work generalises our previous results from \cite{PlunderCoupledSystemsLinear2020} to the case of nonlinear constraints. 
		
	\end{abstract}
	
	\section{Introduction}
	\label{sec:intro}
	
	%
	%
	
	This article is about the question of how to apply the mean-field limit when individual particles are tightly coupled to the dynamics of an macroscopic component. Such a situation arises, for example, in microscopic models for skeletal muscle tissue. In such models, microscopic actin-myosin filaments take the role of \enquote{particles}, and the passive muscle tissue acts like a \enquote{macroscopic component}, as shown in \cref{fig:muscle_macro_particle_perspective}.
	Indeed, in terms of kinetic theory, the established muscle models are macroscopic approximations of the so called sliding filament theory \cite{HuxleyMuscleStructureTheories1957, HuxleyProposedMechanismForce1971, HuxleyChangesCrossStriations1954}.
	The most popular macroscopic approximation is the distributed moment method \cite{ZahalakDistributionMomentApproximation1981} which forms the foundation for most numerical simulations of muscle tissue which include fibers, see for example \cite{BoelMicromechanicalModellingSkeletal2008, GfrererFiberBasedModeling2021, HeidlaufMultiScaleContinuum2016}.
	
 	While sliding filament theory itself is well-established \cite{HerzogSkeletalMuscleMechanics2017, howardMechanicsMotorProteins2001, MaDistributionMomentModel1991, KeenerMathematicalPhysiology.Vol.2009, PhillipsPhysicalBiologyCell2012},
 	a rigorous application of the mean-field limit to realistic models for
 	muscle tissue is currently out of reach.
 	One challenge is to account for the coupling between the microscopic scale (actin-myosin filaments) and the macroscopic components (passive muscle tissue), see also \cite{SimeonModelMacroscaleDeformation2009}.
 	Each deformation of the passive muscle tissue deforms, detaches or breaks attached actin-myosin filaments. In terms of mathematical models, this introduces constraints between the passive muscle tissue and the filaments. To tackle this particular challenge, we will consider a simplified model that focuses solely on the coupling between a macroscopic component and particles. The main result of this article is a rigorous proof of the mean-field limit for such coupled particle systems.

 	\begin{figure}
 		\includegraphics[width=0.8\textwidth]{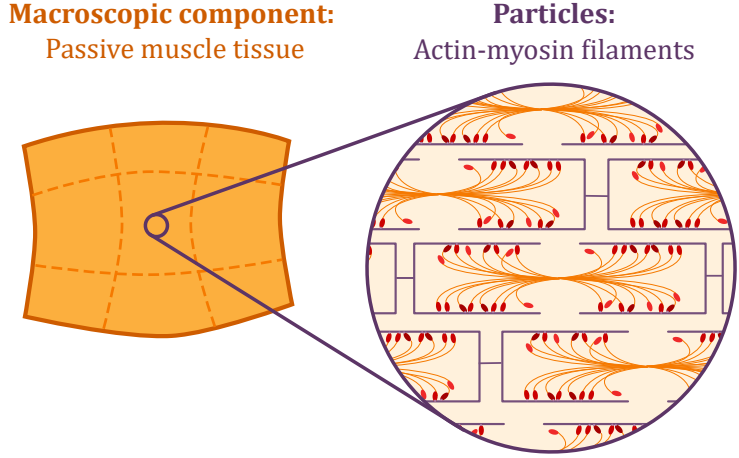}
 		\caption{Skeletal muscle tissue contains arrays of aligned actin-myosin filaments. The myosin heads (red) attach to the actin filament (purple). Millions of such actin-myosin filaments repeatedly perform a cycle of \enquote{attaching, pulling and detaching}. The resulting accumulated force leads to muscle contraction. In article, we ignore the repeated cycling. Instead, we focus on the mathematical implication of the coupling to the macroscopic component.}
 		\label{fig:muscle_macro_particle_perspective}
 	\end{figure} 	

	From a kinetic theory point-of-view, imposing a general constraint on particle positions is not straightforward. Many fundamental results do not generalise in an obvious way to nonlinear position spaces.
	For example, the diffusive scaling limit requires that $(t, \Bf x) \mapsto (\varepsilon t, \varepsilon^2 \Bf x)$ is well-defined for $\varepsilon \to 0$. However, if $\Bfx$ is subject to constraints, this scaling is ill-posed since $\varepsilon^2 \Bf x$ might not satisfy the position constraints.
	Also, other common approaches for deriving macroscopic equations (averaging, moment approximations) are not directly compatible with position constraints. Up to our knowledge, there is no canonical way to rigorously derive macroscopic equations when the particles are subject to general constraints.
	However, the particle model we consider in this article has a specific structure that allows us to derive the mean-field limit and admits macroscopic approximations. We extend the approach from \cite[Section 1.4]{GolseDynamicsLargeParticle2016} to such particle systems with constraints. For further references on kinetic theory, we refer to \cite{CercignaniMathematicalTheoryDilute1994, JabinReviewMeanField2014, DegondMacroscopicLimitsBoltzmann2004, SpohnLargeScaleDynamics1991}.

	\subsection*{\simeon{Major results at a glance}}
		
	Our starting point is a system of differential-algebraic equations (DAE) that describes the motion of particles with positions $\BfX_1, \dots, \BfX_N \in \Rnx$ and a macroscopic component with state $\Bfy \in \Rny$.
	We impose the algebraic constraints
	\[
	g(\BfX_i, \Bfy) = \mathrm{const.} \fai
	\]
	where $g : \Rnx \times \Rny \to \Rnx$ is a nonlinear constraint function. These constraints are uniform, in the sense that all particles have to keep the same constraint function constant, and the constant only depends on the initial condition. 
	The DAE system which we study in this article is the constrained Newton equation
	\begin{subequations}
	\begin{empheq}[left={\empheqbiglbrace~}]{align*}
		m \ddot{\BfX}_i &= F_1(\BfX_i) - \nabla_{\BfX_i} g(\BfX_i,\Bfy) \Bf\lambda_i 
		&\fai,
		\\
		\ddot \Bfy &= F_0(\Bfy) - \frac{1}{N} \sum_{j=1}^N \nabla_{\Bfy} g(\BfX_j,\Bfy) \Bf\lambda_j , &
		\\
		g(\BfX_i,\Bfy) &= g(\init \BfX_i, \init \Bfy)
		&\fai.
	\end{empheq}
\end{subequations}
	A key assumption is that the constraint function $g$ is of full rank with respect to $\BfX_i$.
	Hence, there is a linear map $\Phi(\BfXi,\Bfy)$ which translates the velocity of the macroscopic component to the velocities of the particles, i.e. $\dot \BfXi = \Phi(\BfXi,\Bfy)[ \dot{\Bfy} ]$.
	This relation allows the elimination of the Lagrangian multipliers, and we arrive at the following equivalent ODE model
	\begin{subequations}
	\begin{empheq}[left={\empheqbiglbrace~}]{align*}
		\Big( \frac{1}{N} \sum_{j=1}^N \effO{m}(\BfXj,\Bfy) \Big)  \, \ddot \Bfy &= \frac{1}{N} \sum_{j=1}^N \effO{F}(\BfXj,\Bfy,\dot \Bfy),&
		\\
		\dot \BfXi &= \Phi(\BfXi,\Bfy) [ \dot \Bfy ] &\fai
	\end{empheq}
\end{subequations}
	where $\effO{m}$ and $\effO{F}$ are the effective mass and effective forces of the macroscopic component. The force each particle exerts onto the macroscopic component is included in $\effO{m}$ and $\effO{F}$. 
	
	With the ODE model, we can mostly follow a standard procedure for proving the mean-field limit. 
	The mean-field PDE is given by
	\begin{subequations}
	\begin{empheq}[left={\empheqbiglbrace~}]{align*}
		\Big( \int \effO{m}(x,\Bfy) f(x,t) \, \dif x
		\Big) \, \ddot \Bfy &= \int \effO{F}(x,\Bfy,\dot \Bfy) f(x,t) \, \dif x,
		\\
		\partial_t f(x,t) + \mathrm{div}_x \big( f(x,t) \, \Phi(x,\Bfy)[\dot \Bfy] \big) &= 0
	\end{empheq}
\end{subequations}
	where $f(x,t)$ is the density of the particles with the informal relation $\BfXi(t) \sim f(x,t) \dif x$.
	A difference to classical mean-field equations is the occurrence of a non-constant \emph{mean-field mass} in addition to the classical mean-field force. Dealing with the mean-field mass is the primary difference compared to standard procedures for proving the mean-field limit.
	
	Our main results are the proof of well-posedness and a stability estimate.
	Under some assumptions, we show well-posedness of the ODE model and the characteristic flow of the mean-field PDE.
	To obtain sufficient Lipschitz bounds, we use conservation of total energy.
	To show the mean-field limit, we prove a generalisation of Dobrushin's stability estimate \cite[Section 1.4]{GolseDynamicsLargeParticle2016}, which implies
	\begin{align*}
		W_1(\mu^t_1,\mu^t_2) &+ \Vert \Bfy_1 + \Bfy_2 \Vert + \Vert \dot{\Bfy}_1 + \dot{\Bfy}_2 \Vert \\
		&\leq C e^{Lt} \big( W_1(\init \mu_1,\init \mu_2) + \Vert \init{\Bfy}_1 - \init \Bfy_2 \Vert + \Vert \init{\Bfv}_1 - \init{\Bfv}_2 \Vert \big)
	\end{align*}
	where $W_1$ is the Monge-Kantorovich distance with exponent $1$ (also called Wasserstein distance), $\init \mu_1, \init \mu_2$ are two different initial particle distributions
	and $\mu^t_1, \mu^t_2$ are weak solutions of the mean-field PDE.
	The stability estimate implies convergence in the mean-field limit and provides a rate of convergence as well.

	This article generalises the results from our previous work \cite{PlunderCoupledSystemsLinear2020}.
	Instead of assuming that the constraints and forces are linear, we allow nonlinear constraints and forces.
	This is motivated by the application to models for muscle tissue \cite[Section 5.2]{PlunderCoupledSystemsLinear2020}.
	However, the nonlinear constraints introduce two new challenges.
	First, the position space becomes nonlinear, which prevents many simplifications.
	Second, the mean-field mass is non-constant in contrast to the linear case. In the context of the Newton equations of the macroscopic component, this implies that we need sufficient bounds for the inverse of the mean-field mass.
	
	We want to point out that even with the extension to nonlinear constraints, the applicability of particle systems with uniform, full-rank constraints remains limited. Practical applications would require further extensions, which are beyond the scope of this article. For example, in the context of muscle tissue models, a necessary extension would be to allow particles to switch between a constrained and an unconstrained state. With such a mechanism, one could model the essential driver for muscle contraction, which is the attachment and detachment of actin-myosin filaments, see  \cite[Section 5.3]{PlunderCoupledSystemsLinear2020}. For the case of unconstrained particles, related mean-field limit results already exist  \cite{DarlingDifferentialEquationApproximations2008, DavisPiecewiseDeterministicMarkov1984}, but these results do not directly apply to constrained particles.
	
	This article is organised as follows:
	\Cref{sec:discrete_model} defines the discrete particle model as a DAE model and an equivalent ODE model. It shows conservation of energy and well-posedness for the ODE model.
	\Cref{sec:mean_field} starts with a formal derivation of the mean-field equations for the ODE model.
	It continues with showing that the mean-field equations are consistent with the ODE model, preserve energy and are well-posed. As a main result, it proves a stability estimate that implies convergence of the mean-field limit.
	\Cref{sec:appendix} contains the proof of the equivalence of the DAE model and the ODE model.

	\clearpage

	\section{The discrete particle model with uniform full-rank constraints}
	\label{sec:discrete_model}
	This section describes the discrete particle model with uniform, full-rank constraints. We start with a DAE model and derive an equivalent ODE model for the particle dynamics. Then, we prove conservation of energy and well-posedness for the ODE model.
	
	The derivation of an equivalent ODE model has the following advantage. If we apply the mean-field limit directly to the DAE model, the resulting mean-field equation would be a partial differential-algebraic equation (PDAE). This would be an unnecessary complication. Using the ODE model avoids this issue. By exploiting the particular structure of the DAE, we can eliminate the constraints and their Lagrangian multipliers. The resulting equations will be an equivalent ODE model which yields a PDE in the mean-field limit.

	\subsection{The DAE model}
	\label{sec:dae_model}
	
	Let $\BfCalXN \coloneqq (\BfX_1,\dots,\BfX_N) \in (\RR^{n_x})^N$ denote the positions of $N$ identical, non-interacting, microscopic particles and let $\Bfy \in \RR^{n_y}$ denote the state of a macroscopic component.
	We denote the internal forces of both systems as
	\begin{align*}
		F_0(\Bfy) \coloneqq -\nabla_{\Bfy} \Cal W_0(\Bfy)
		\quad
		\text{and}
		\quad
		F_1(\BfXi) \coloneqq -\nabla_{\BfXi} \Cal W_1(\BfXi) \fai
	\end{align*}
	where $\Cal W_0 \in C^1(\Rny,\RR)$ and $\Cal W_1 \in C^1(\Rnx,\RR)$ are the potential energies of the macroscopic component and the microscopic particles.
	
	The central connection between the particles and the macroscopic component are 
	the constraints
	\begin{align*}
		g(\BfXi(t), \Bfy(t)) = g(\init \BfXi, \init \Bfy) \fai \text{~and~} t \geq 0
	\end{align*}
	where $g \in C^2(\Rnx \times \Rny,\Rnx)$ is the constraint function and $\init \BfXi \in \Rnx, \init \Bfy \in \Rny$ are the initial conditions.
	We call these constraints uniform since all particles have to remain in the level set of the same constraint function.
	Moreover, we assume that the constraint function has full-rank, i.e.
	\begin{align}
		\nabla_\BfX g(\BfX, \Bfy) \quad \text{is invertable for all} ~ \BfX \in \Rnx, \Bfy \in \Rny.
	\end{align}
	This implies that the state of the macroscopic component determines locally the position of all particles.
	
	Now, we can state the discrete particle model as a constrained Newton equation
	\begin{subequations}
	\begin{empheq}[left={\empheqbiglbrace~}]{align}
			\label{eq:x_micro_ind3}
			m \ddot{\BfX}_i = F_1(\BfX_i) - \nabla_{\BfX_i} g(\BfX_i,\Bfy) \Bf\lambda_i 
			& \quad \fai,
			\\
			\label{eq:y_micro_ind3}
			\ddot \Bfy = F_0(\Bfy) - \frac{1}{N} \sum_{j=1}^N \nabla_{\Bfy} g(\BfX_j,\Bfy) \Bf\lambda_j , &
			\\
			\label{eq:g_micro_ind3}
			g(\BfX_i,\Bfy) = g(\init \BfX_i, \init \Bfy)
			& \quad \fai
	\end{empheq}
	\end{subequations}
	with the initial conditions
	\begin{align}
		\label{eq:micro_init_data}
		\BfXi(0) = \init \BfXi, \quad
		\Bfy(0) = \init  \Bfy
		\quad
		\text{and}
		\quad
		\dot \Bfy(0) = \init \Bfv
		\fai.
	\end{align}
	The equations \crefDAE form a system of differential-algebraic equations (DAEs). Therefore, we will refer to it as the \textit{DAE model}.	
	The algebraic variables $\Bf \lambda_1, \dots, \Bf \lambda_N \in \RR^{n_x}$ are the Lagrangian multipliers for \cref{eq:g_micro_ind3}, and their initial value is determined implicitly by the system.
	
	The main objective of this article is to study the limit $N \to \infty$ for this system.
	The scaling factor $\frac{1}{N}$ represents the mean-field scaling and ensures that the total energy remains bounded in the limit $N \to \infty$. This factor can be derived by the choice of an appropriate time-scale, see \cite[Section 1.2]{GolseDynamicsLargeParticle2016}.

\subsection{The ODE model}
	\label{sec:ode_model}
	The derivation of an equivalent ODE model is elementary but lengthy. Therefore, we just present the results here, whereas the full computations can be found in \cref{sec:appendix}. 
	
	Due to the full-rank constraint, the velocity of the macroscopic component $\dot \Bfy$ determines the velocity of each particle $\dot \BfXi$, namely
	\begin{align}
		g = \text{const.} \Leftrightarrow \frac{\dif g}{\dif t} = 0 \Leftrightarrow \quad \dot \BfXi = \Phi_i[ \dot \Bfy]
		\label{eq:Xdot_ode_tmp}
	\end{align}
	where \[
	\Phi_i \coloneqq
	\Phi(\BfXi,\Bfy) 
	\coloneqq
	-\left( \DifXj g(\BfXi,\Bfy) \right)^{-1} 
	\Dify g(\BfXi,\Bfy) \in \RR^{n_x \times n_y}.
	\]
	
	Consequentially, 
	\begin{align}
		\frac{\dif g}{\dif t} = 0 \Rightarrow \frac{\dif^2 g}{\dif t^2} = 0  \Leftrightarrow
		 \quad  \ddot \BfXi = \Phi_i[ \ddot \Bfy ] + \Omega_i [\dot \Bfy, \dot \Bfy]
		 \label{eq:Xddot_ode_tmp}
	\end{align}
	where	
	\begin{align*}
		\Omega_i[\Bf v, \Bf w] 
	&\coloneqq \Omega(\BfXi, \Bfy)[\Bf v, \Bf w]
	\notag \\
	&\coloneqq \DifXi \Phi(\BfXi,\Bfy)[ \Bf v, \Phi(\BfXi,\Bfy)[ \Bf w ] ] + \Dify \Phi(\BfXi,\Bfy)[ \Bf v, \Bf w]
\end{align*}
for $\Bf v, \Bf w \in \Rny$. We remark that taking derivatives of the algebraic equation $g = 0$ is related to the concept of index reduction \cite[Chap VII]{HairerSolvingOrdinaryDifferential2010}.
	
	With the relations \cref{eq:Xdot_ode_tmp,eq:Xddot_ode_tmp}, we can solve \cref{eq:x_micro_ind3} for the Lagrangian multipliers and then eliminate both $\Bf \lambda_i$ and $\ddot \BfXi$ from the system \crefDAE to obtain an equivalent ODE model
	\begin{subequations}
		\begin{empheq}[left={\empheqbiglbrace~}]{align}
			\effN{m} \ddot \Bfy &= \effN{F}, 
			\label{eq:y_micro_ode}\\
			\dot \BfXi &= \Phi_i [ \dot \Bfy ], \fai
			\label{eq:x_micro_ode}
		\end{empheq}
	\end{subequations}
	with initial condtions
	\begin{align}
		\Bfy(0) = \init \Bfy,
		\quad
		\dot \Bfy(0) = \init \Bfv
		\quad \text{and} \quad
		\BfXi(0) = \init \BfXi \fai
	\end{align}
	where the effective mass and effective force are
	\begin{align}
		\effN{m}(\BfCalXN, \Bfy) &= (I_{\Rny} + \frac{m}{N} \sum_{j=1}^N \Phi_j^T \Phi_j ), 
		\label{eq:def_effNm}\\
		\effN{F}(\BfCalXN, \Bfy, \dot \Bfy) &= 
		F_0(\Bfy) 
		+ 
		\frac{1}{N} \sum_{j=1}^N \Phi_j^T  (   
		F_1(\BfXj) - m \Omega_j [ \dot \Bfy , \dot \Bfy ])
		\label{eq:def_effNF}.
	\end{align}
	The super-index ${}^{(N)}$ identifies the number of particles. For a detailed derivation of \crefODE, we refer to \cref{sec:appendix}.

	The system \crefODE is a reduced form of the DAE model \crefDAE. In the following, we will call \crefODE the \emph{ODE model}. 
	We notice that \cref{eq:y_micro_ode} is in the form of Newton's second law. Therefore, we identify $\effN{m}$ as the effective mass of the macroscopic component and $\effN{F}$ as the effective force.

	The equations \cref{eq:def_effNm,eq:def_effNF} fit well into the setting of mean-field theory since both terms are mean values of the individual influence of all particles, i.e.
	\begin{align}
		\effN{m}(\BfCalXN, \Bfy) &= \frac{1}{N} \sum_{j=1}^N \effO{m}(\BfXj,\Bfy), \\
		\effN{F}(\BfCalXN, \Bfy, \dot \Bfy) &= \frac{1}{N} \sum_{j=1}^N \effO{F}(\BfXj,\Bfy, \dot \Bfy)
	\end{align}

	Every solution of the DAE model \crefDAE also solves the ODE model \crefODE, and with solutions of the ODE model we can recover the Lagrangian mutlipliers to construct a solution of the DAE model. The next proposition formalises this equivalence.
	
	\begin{proposition}
		\label{prop:ode_dae}
		Let $g \in C^2(\Rnx \times \Rny, \Rnx)$ be a constraint function such that
		\begin{align}
			\label{eq:prop_g_inv}
		\nabla_{\BfX} g(\BfX,\Bfy)) \text{~is invertable for all~} \BfX \in \Rnx, \Bfy \in \Rny.
		\end{align}
		
		If 
		\begin{align*}
			\Bfy \in C^2([0,T],\Rny), \quad
			\BfXi \in C^1([0,T],\Rnx) \fai
		\end{align*}
		is a solution of the ODE model \crefODE,
		then $\Bf\lambda_i : [0,T] \to \Rnx$ defined by \cref{eq:app_lambda_def} is a continuous function and $(\Bfy, \BfX_1,\dots,\BfX_N,\Bf\lambda_1,\dots,\Bf\lambda_N)$ is a solution of the DAE model \crefDAE.
		
		Vice versa, if 
		\begin{align*}
			\Bfy \in C^2([0,T],\Rny), \quad
			\BfXi \in C^2([0,T],\Rnx), \quad
			\Bf\lambda_i \in C^0([0,T],\Rnx) \fai 
		\end{align*} 
		is a solution of the DAE model \crefDAE, then $(\Bfy, \BfX_1,\dots,\BfX_N)$ is a solution of the ODE model \crefODE.
	\end{proposition}

	The proof is given in \cref{sec:appendix}.
	
	\subsection{Assumptions}
	\label{sec:assumptions}
	In the following, we will collect sufficient assumptions for well-posedness of the ODE model \crefODE and for convergence in the mean-field limit.
		
	For the terms of the ODE model \crefODE we assume:
	\begin{enumerate}
		\item The forces $F_0 : \Rny \to \Rny$ and $F_1 : \Rnx \to \Rnx$ are Lipschitz continuous and bounded, i.e. there exists constants $L_0, L_1, M_0, M_1 \geq 0$ such that 
		\begin{align}
			\label{ass:L_F0} \tag{A1a}
			\Vert F_0(\Bfy) - F_0(\Bfy') \Vert 
			&\leq
			L_0 \Vert \Bfy - \Bfy' \Vert,\\
			\label{ass:M_F0} \tag{A1b}
			\Vert F_0(\Bfy) \Vert 
			&\leq M_0, \\
			\label{ass:L_F1} \tag{A1c}
			\Vert F_1(\BfX) - F_1(\BfX') \Vert 
			&\leq
			L_1 \Vert \BfX - \BfX'\Vert, \\
			\label{ass:M_F1} \tag{A1d}
			\Vert F_1(\BfX) \Vert 
			&\leq M_1.
		\end{align}
		for all $\BfX, \BfX'\in \Rnx$, $\Bfy, \Bfy'\in \Rny$.
		
		\item The constraint 
		\begin{align}
			g : \Rnx \times \Rny \to \Rnx \text{~is twice continuously differentiable}
			\label{ass:g_C2} \tag{A2a}
		\end{align}
		 and
		\begin{align}
			\label{ass:g_invertible} \tag{A2b}
			\DifX g(\BfX, \Bfy) \text{~ is invertible}
		\end{align}
		for all $\BfX \in \Rnx$ and $\Bfy \in \Rny$.
		\item The map 
		\begin{align*}
			\Phi 
			: \Rnx \times \Rny \to \RR^{n_x \times n_y}
			: (\BfX, \Bfy) \mapsto -\left(\DifX g(\BfX,\Bfy)\right)^{-1} \Dify g(\BfX,\Bfy)
		\end{align*}
		is Lipschitz continuous and bounded, i.e. there are constants $L_\Phi, M_\Phi \geq 0$ such that
		\begin{align}
			\label{ass:L_phi} \tag{A3a}
			\Vert \Phi(\BfX,\Bfy) - \Phi(\BfX', \Bfy') \Vert 
			&\leq
			L_\Phi \left( \Vert \BfX - \BfX'\Vert + \Vert \Bfy - \Bfy' \Vert \right) \\
			\label{ass:M_phi} \tag{A3b}
			\Vert \Phi(\BfX, \Bfy) \Vert 
			&\leq M_\Phi.
		\end{align}
		for all $\BfX, \BfX'\in \Rnx$ and $\Bfy, \Bfy'\in \Rny$.
		\item The map 
		\begin{align*}
			\Omega &: \Rnx \times \Rny \to \RR^{n_x \times n_y \times n_y} \\ 
			&:
			(\BfX, \Bfy) \mapsto \DifX \Phi(\BfX, \Bfy) \cdot \Phi(\BfX,\Bfy) + \Dify \Phi(\BfX,\Bfy)
		\end{align*}
		is Lipschitz continuous and bounded, i.e. there exist constants $L_\Omega, M_\Omega \geq 0$ such that
		\begin{align}
			\label{ass:L_omega} \tag{A4a}
			\Vert \Omega(\BfX,\Bfy) - \Omega(\BfX', \Bfy') \Vert 
			&\leq
			L_\Omega \left( \Vert \BfX - \BfX'\Vert + \Vert \Bfy - \Bfy' \Vert \right) \\
			\label{ass:M_omega} \tag{A4b}
			\Vert \Omega(\BfX, \Bfy) \Vert 
			&\leq M_\Omega
		\end{align}
		for all $\BfX, \BfX'\in \Rnx$ and $\Bfy, \Bfy'\in \Rny$.
	\end{enumerate}

	\begin{remark}[Simplified assumptions for the constraint]
		Checking the assumptions \crefrange{ass:L_phi}{ass:M_omega} is a tedious task.
		A variant of \cref{lem:m_inv_f} yields the following stronger  but simplified assumptions.
		
		If $g$ and its first and second derivatives are bounded, Lipschitz continuous and $\nabla_\BfX g$ is uniformly elliptic, i.e. there exist a constant $\delta > 0$ such that
		\begin{align}
		v^T \nabla_\BfX g(\BfX,\Bfy) v \geq \delta \Vert v \Vert^2
		\quad 
		\text{for all~} v \in \Rnx \text{~and~} (\BfX,\Bfy) \in \Rnx \times \Rny,
		\end{align}
		then, $g$ satisfies \crefrange{ass:L_phi}{ass:M_omega}.
		
	\end{remark}
		
	\subsection{Well-posedness of the ODE model}
	
	The assumptions \crefrange{ass:L_F0}{ass:M_omega} are chosen such that the ODE model \crefODE is globally Lipschitz. To show this, the main task is to prove that
	${\effN{m}}^{-1} \effN{F}$ 
	is Lipschitz continuous. We will do this in two steps.
	
	First, we show in \cref{lem:ODE_conserves_energy} that for given initial conditions, there exists an upper bound for the speed of the macroscopic component, i.e.
	\[
	\Vert \dot \Bfy(t) \Vert \leq M_v \quad \text{for all~} t \geq 0
	\]
	where $M_v > 0$ is a constant depending only on the initial conditions. 
	Then, we can show that 	$\effN{m}$ and $\effN{F}$ are bounded and Lipschitz continuous.
	
	In the second step, we use that $\effN{m}$ is also symmetric and uniformly elliptic to show that the product ${\effN{m}}^{-1} \effN{F}$ is Lipschitz continuous as well.
	
	Finally, the Picard-Lindelöf Theorem implies well-posedness of \crefODE.
	
	\begin{lemma}[Conservation of energy]
		\label{lem:ODE_conserves_energy}
		For $T > 0$, let $(\BfCalXN,\Bfy)$ be a solution of \crefODE with $\Bfy \in C^2([0,T],\Rny)$ and $\BfCalXN \in C^1([0,T],(\Rnx)^N)$.
		Then, the total energy
		\begin{align}
			\label{eq:ode_total_energy}
			E = 
			\frac{1}{2} \Vert \dot \Bfy \Vert^2  
			+ \mathcal{W}_0(\Bfy)
			+ \frac{1}{N} \sum_{i=1}^N 
			\left(
			\frac{m}{2} \Vert \dot \BfXi \Vert^2 + \mathcal{W}_1(\BfX_i) 
			\right)
		\end{align}
		is a conserved quantity. In particular, 
		\begin{align}
			\Vert \dot \Bfy(t) \Vert \leq 2 \left( E(\init \BfCalX, \init \Bfy, \init \Bfv) \right)^2
			\quad  \text{for all~} t \in [0,T].
			\label{eq:v_bounded}
		\end{align}
	\end{lemma}
	
	\begin{proof}
		\label{proof:cons_energy_ode}
		
		We recall that the time derivative of \cref{eq:x_micro_ode} yields
		\begin{align}
			\label{eq:tmp_x_accel}
			\ddot \BfXi = \Phi_i[ \ddot \Bfy ] + \Omega_i[\dot \Bfy, \dot \Bfy].
		\end{align}
	
		Using \cref{eq:x_micro_ode,eq:tmp_x_accel,eq:y_micro_ode}, we compute the time derivative of the total energy, 
		\begin{align*}
			\dot E
			&= \dot \Bfy^T ( \ddot \Bfy - F_0(\Bfy) )
			+ \frac{1}{N} \sum_{i=1}^N \dot{\BfXi}^T ( m \ddot \BfXi - F_1(\BfXi) )
			\\
			&= \dot \Bfy^T ( \ddot \Bfy - F_0(\Bfy) )
			+ \dot \Bfy^T ( \frac{1}{N} \sum_{i=1}^N \Phi_i^T ( m (\Phi_i[ \ddot \Bfy ] + \Omega_i[ \dot \Bfy, \dot \Bfy ]) - F_1(\BfXi) ) )
			\\
			&= \dot \Bfy^T \left( 
			(1+ \frac{1}{N} \sum_{i=1}^N m\Phi_i^T \Phi_i) \ddot \Bfy
			- F_0(\Bfy) - \frac{1}{N} \sum_{i=1}^N \Phi_i^T (F_1(\BfXi) - m \Omega_i[\dot \Bfy, \dot \Bfy]) \right) \\
			&= \dot \Bfy^T ( \effN{m} \ddot \Bfy - \effN{F} )\\
			&= 0.			
		\end{align*}
	\end{proof}	
		
	 \cref{lem:ODE_conserves_energy} justifies the assumption that the velocity $\Bfv = \dot \Bfy$ is bounded, which is useful for the next lemma.
	\begin{lemma}
		\label{lem:effOM_effOF}
		For any constant $M_v \geq 0$, the maps
		\[
		(\BfX,\Bfy) \mapsto \effO{m}(\BfX,\Bfy)
		\quad \text{and} \quad
		(\BfX,\Bfy,\Bfv) \mapsto \effO{F}(\BfX,\Bfy,\Bfv)
		\]
		and
		\[
		(\BfCalXN,\Bfy) \mapsto \effN{m}(\BfCalXN,\Bfy)
		\quad \text{and} \quad
		(\BfCalXN,\Bfy,\Bfv) \mapsto \effN{F}(\BfCalXN,\Bfy,\Bfv)
		\]
		are bounded and Lipschitz continuous for all
		$\BfX \in \Rnx, \BfCalXN \in (\Rnx)^N$ and for all $\Bfy, \Bfv \in \Rny$ with $\Vert \Bfv \Vert \leq M_v$.
	\end{lemma}
	
	\begin{proof}
		
		In the following, we list which assumptions imply boundedness and Lipschitz continuity for the terms of $\mfO{m}, \mfO{F}$. We recall
		\begin{align*} 
		\effO{m}(\BfX,\Bfy) 
			&= I_\Rny + m \Phi(\BfX,\Bfy)^T \Phi(\BfX,\Bfy), \\
		\effO{F}(\BfX,\Bfy,\Bfv)
			&= F_0(\Bfy) + \Phi(\BfX,\Bfy)^T F_1(\BfX) + m \Omega(\BfXi,\Bfy)[\Bfv,\Bfv]. \\
		\end{align*}
		We use the general fact that if $h_1, h_2$ are both bounded and Lipschitz continuous, then the product $h_1 h_2$ is also bounded and Lipschitz continuous.
		
		\begin{itemize}
			\item $\Phi$ is bounded and Lipschitz by \cref{ass:L_phi,ass:M_phi}, therefore, $m \Phi^T \Phi$ is bounded and Lipschitz as well.
			\item $\Phi$ and $F_1$ are bounded and Lipschitz continuous by \cref{ass:L_F1,ass:M_F1,ass:L_phi,ass:M_phi}, therefore $\Phi(\BfX,\Bfy)^T F_1(\BfX)$ is bounded and Lipschitz.
			\item $\Phi$ and $\Omega$ are bounded and Lipschitz continuous by \cref{ass:L_phi,ass:M_phi,ass:L_omega,ass:M_omega}, therefore, $m \Omega(\BfXi,\Bfy)[\Bfv,\Bfv]$ is bounded and Lipschitz continuous provided $\Vert \Bfv \Vert \leq M_v$.
			\item $F_0$ and $F_1$ are bounded and Lipschitz by \cref{ass:L_F0,ass:M_F0,ass:L_F1,ass:M_F1}.
		\end{itemize}
		The statement for $\effN{m}$ and $\effN{F}$ follows since they are the mean of the bounded and Lipschitz continuous functions $\mfO{m}$ and $\mfO{F}$.
	\end{proof}
	
	Next, we show that ${\effN{m}}^{-1} \effN{F}$ is bounded and Lipschitz continuous. For this task, we prove a slightly more general lemma, which we can reuse later to show well-posedness of the mean-field PDE.
	We introduce the following notation. Let $(M,d_M)$ be a complete metric space, $n \in \mathbb{N}$ and $h: M \to \RR^n$, then we define
	\[
	\Op{Lip}(h) \coloneqq \sup_{\substack{a, b \in M\\ a \neq b}} \frac{\Vert h(a) - h(b) \Vert}{d_M(a, b)}.
	\]
	\newcommand{\Lip}{\mathrm{Lip}}
	
	\begin{lemma}
		\label{lem:m_inv_f}
		Let $(M, d_M)$ be a complete metric space and $n \in \mathbb{N}$. Let $m : M \to \RR^{n \times n}$ and $h: M \to \RR^n$ be bounded and Lipschitz continuous. Moreover, let $m(z)$ be symmetric and uniformly elliptic for all $z \in M$, i.e. there exists a positive constant $\delta_m > 0$ such that $v^T m(z) v^T \geq \delta_m \Vert v \Vert^2$ for all $z, v \in \RR^{n}$ . 
		
		Then, the map
		\[
		M \to \RR^n : z \mapsto  (m(z))^{-1} h(z)
		\]  
		is Lipschitz continuous with Lipschitz constant $\delta_m^{-2} \mathrm{Lip}(m) \Vert h \Vert_\infty + \Vert m \Vert_\infty \mathrm{Lip}(h)$.
	\end{lemma}

	\begin{proof}
		For arbitrary $z \in M$, the ellipticity implies  that the lowest eigenvalue of $m(z)$ is bounded from below by $\delta_m > 0$. Hence, the inverse $(m(z))^{-1}$ exists, and its norm is bounded by
		\[
		\Vert (m(z))^{-1} \Vert \leq \delta_m^{-1}. 
		\]
		Next, we consider the map $\Inv : GL(n) \to GL(n)$ which has the derivative 
		\[
		\partial_z \Inv(m(z))[w] = -m(z)^{-1} w \, m(z)^{-1} \quad \text{for~} w \in \RR^{n \times n},
		\]
		which implies 
		\[
		\Vert \partial_z \mathrm{\Inv}(m(z)) \Vert \leq \delta_m^{-2}.
		\]		
		Now, we can show that $z \mapsto \Inv(m(z))$ is Lipschitz. For $z_1, z_2 \in M$, the matrices $m(z_1), m(z_2)$ have only positive eigenvalues. Hence, they are part of the same connected component of $GL(d)$.
		Therefore, the mean-value theorem is applicable and yields
		\begin{align*}
			\Vert \Inv(m(z_1)) - \Inv(m(z_2)) \Vert \leq \delta_m^{-2} \Vert m(z_1) - m(z_2) \Vert \leq \delta_m^{-2} \Lip(m) d_M(z_1, z_2).
		\end{align*}		
		The claim follows from the computation
		\begin{align*}
			\MoveEqLeft[4]
			\Vert \Inv(m(z_1)) h(z_1) - \Inv(m(z_2)) h(z_2) \Vert \\ 
			&\leq
			\Vert \Inv(m(z_1)) h(z_1) - \Inv(m(z_2)) h(z_1) \Vert 
			\\
			&\quad+
			\Vert \Inv(m(z_2)) h(z_1) - \Inv(m(z_2)) h(z_2) \Vert 
			\\
			&\leq
			\delta_m^{-2} \Lip(m) d_M(z_1, z_2) \Vert \Vert h \Vert 
			+
			\Vert \Inv(m(z_2)) \Vert \Lip(h) d_M(z_1, z_2). 
		\end{align*}		
	\end{proof}		
			
	\begin{lemma}
		\label{lem:effM_inv_effF_lipschitz}
		For any constant $M_v \geq 0$, the map
		\[
		(\BfCalXN,\Bfy,\Bfv) \mapsto \left( \effN{m}(\BfCalXN,\Bfy) \right)^{-1} \effN{F}(\BfCalXN,\Bfy,\Bfv)
		\]
		is Lipschitz continuous for all
		$\BfCalXN \in (\Rnx)^N$ and all $\Bfy, \Bfv \in \Rny$ with $\Vert \Bfv \Vert \leq M_v$.
	\end{lemma}
	
	\begin{proof}
		By \cref{lem:effOM_effOF}, we know that $\effN{m}$ and $\effN{F}$ are  Lipschitz continuous and bounded.
		To apply \cref{lem:m_inv_f}, it is left to show that $\effN{m}$ is symmetric and uniformly elliptic. It suffices to show these properties for $\effO{m}$.
		Symmetry follows directly from the definition $\effO{m}(\BfX,\Bfy) = I_{\Rny} + m \Phi(\BfX,\Bfy)^T \Phi(\BfX,\Bfy)$,
		and since $\Phi(\BfX,\Bfy)^T \Phi(\BfX,\Bfy)$ is non-negative definite, we can conclude that $\Bf w^T \effO{m} \Bf w \geq \Vert \Bf w \Vert^2$ for all $\Bf w \in \Rny$.
		Hence, by \cref{lem:m_inv_f}, $(\effN{m})^{-1} \effN{F}$ is Lipschitz continuous provided $\Vert \Bfv \Vert \leq M_v$.
	\end{proof}
			
	\begin{theorem}[Existence and uniqueness of solutions]
		For $T > 0$ and for initial conditions $\init \BfCalX \in (\Rnx)^N, \init \Bfy, \init \Bfv \in \Rny$ the ODE model \crefODE has a unique solution $(\BfCalXN,\Bfy)$ with $\BfCalXN \in C^2([0,T],(\Rnx)^N)$ and $\Bfy \in C^2([0,T],\Rny)$.
	\end{theorem}

	\begin{proof}
		First, we rewrite the ODE model \crefODE as
	\begin{subequations}
	\begin{empheq}[left={\empheqbiglbrace~}]{align}
			\ddot \Bfy &= (\effN{m}(\BfCalXN,\Bfy))^{-1} \effN{F}(\BfCalXN,\Bfy,\dot \Bfy),&
			\label{eq:tmp_ode_y}
			\\
			\dot \BfXi &= \Phi(\BfXi,\Bfy)[ \dot \Bfy ]
			&\fai.
			\label{eq:tmp_ode_x}
		\end{empheq}
	\end{subequations}
		
		We define $M_v \coloneqq 2 (E(\init \BfCalX,\init \Bfy,\init \Bfv))^2$ where $E$ is the total energy from \cref{eq:ode_total_energy}.
		
		The assumptions \cref{ass:L_phi,ass:M_phi} imply that 
		\[
		(\BfX,\Bfy,\Bfv) \mapsto \Phi(\BfX,\Bfy)[\Bfv]
		\]
		is Lipschitz continuous for $\Vert \Bfv \Vert \leq M_v$ and \cref{lem:effM_inv_effF_lipschitz}
		shows that $(\effN{m})^{-1} \effN{F}$ is Lipschitz continuous for $\Vert \Bfv \Vert \leq M_v$.

		 Existence and uniqueness of solutions for \crefODE follows from the Picard-Lindelöf Theorem and \cref{lem:ODE_conserves_energy}.
		 
		 The trajectory $\BfCalXN$ is actually twice continuously differentiable, since the right-hand side of \cref{eq:tmp_ode_x} is continuously differentiable by assumption \cref{ass:g_C2}.
	\end{proof}

	\section{The mean-field limit for the ODE model}
	\label{sec:mean_field}
	
	For an introduction to mean-field limits, we recommend the lecture notes from François Golse \cite{GolseDynamicsLargeParticle2016} or the review \cite{JabinReviewMeanField2014}.
	We will adapt the approach from \cite[Section 1.4]{GolseDynamicsLargeParticle2016} to prove the mean-field limit for the ODE model \crefODE. The main difference to \cite[Section 1.4]{GolseDynamicsLargeParticle2016} is that our ODE system contains an additional mean-field mass term, which requires additional estimates.
	First, we formally define a set of equations and then prove that these equations are the ODE model's mean-field limit. Our main result is the stability estimate in \cref{thm:dobrushin}.
	
	\subsection{Formal mean-field limit for the ODE model}
	
	We use the following notation from probability theory, see \cite{GolseDynamicsLargeParticle2016}. The space of probability measures over $\Rnx$ is $\PRnx{}$ and the space of probability measures with finite first moment is 
	\[
	\PRnx{1} \coloneqq \left\{ \mu \in \PRnx{} \mid \int \Vert x \Vert \, \dif \mu(x) < \infty \right\}.
	\]
	The push-forward of a measure $\mu \in \PRnx{1}$ under a map $\varphi : \Rnx \to \Rnx$ is defined as
	\begin{align*}
		\varphi \# \mu (A) \coloneqq \mu( \varphi(A) ) 
		\quad 
		\text{for all~} A \in \mathfrak{B}(\Rnx) 
	\end{align*}
	where $\mathfrak{B}(\Rnx)$ are the Borel measurable sets in $\Rnx$.

	Let $(\mu^t)_{t \in [0,\infty)} \subset \mathcal{P}^1(\RR^{n_x})$ be a family of probability measures representing the statistical particle distribution and let $\init \mu \in \PRnx{1}$ be an initial particle distribution.

	The formal assumption of the mean-field equations is as follows. If we consider independent, random initial conditions $\init \BfX_1, \dots, \init \BfX_N \sim \init \mu$, then, we can hope that there exists a statistical distribution $\mu^t \in \PRnx{1}$ such that
	$\BfX_1(t),\dots,\BfX_N(t) \sim \mu^t$ for $t \geq 0$.\footnote{For general interacting particle systems, this property is non-trivial and relates to the concept of propagation of chaos \cite{JabinReviewMeanField2014}. However, this property is less surprising for the ODE model \crefODE since it lacks pairwise interaction between particles.}
	This will be made precise in \cref{lem:mean_field_consistency} and finally with the proof of mean-field convergence. However, for now, we continue with the formal argumentation.

	We define the mean-field characteristic flow 
	\[
	(X^t)_{t \in [0,\infty)} \text{~with~} X^t : \Rnx \to \Rnx
	\]
	which describes the trajectories of a single \enquote{virtual} particle subject to the\\ mean-field influence of particles distributed according to the particle distribution $(\mu^t)_{t \in [0,\infty)}$.	
	It is governed by the \emph{mean-field characteristics flow equations} 	
	\newcommand{\crefCharFlowODE}{{\cref{eq:meso_y_eff,eq:meso_X_flow,eq:meso_mu_pushforward}}\xspace}
	\begin{subequations}
		\begin{empheq}[left={\empheqbiglbrace~}]{align}
			\label{eq:meso_y_eff}
			\eff{m}(\mu^t,\Bfy) \, \ddot{\Bfy} 
			&= 
			\eff F(\mu^t,\Bfy,\dot \Bfy), \\
			\label{eq:meso_X_flow}
			\dot X^t(x) &= \Phi(X^t(x),\Bf y)[ \dot \Bfy ] \quad
			\text{for all~} x \in \RR^{n_x}, \\
			\label{eq:meso_mu_pushforward}
			\mu^t &= X^t \# \init{\mu}
		\end{empheq}
	\end{subequations}
	and initial conditions
	\begin{align*}
		\mu^0 = \init \mu,
		\quad
		\Bfy(0) = \init \Bfy, 
		\quad
		\dot \Bfy(0) = \init \Bfv
		\quad
		\text{and}
		\quad
		X^0(x) = x \quad \text{for all~} x \in \RR^{n_x}
	\end{align*}
	with the definitions
	\begin{subequations}
		\begin{align}
			\eff{m}(\mu, \Bfy) 
			&\coloneqq 
			\int_{\Rnx} \effO{m}(x,\Bfy) \, \dif \mu(x),
			\\			
			\eff{F}(\mu, \Bfy, \Bfv) 
			&\coloneqq 
			\int_{\Rnx} \effO{F}(x,\Bfy,\Bfv) \, \dif \mu(x).
		\end{align}
	\end{subequations}
	Both integrals are finite, since both integrands are Lipschitz continuous by \\\cref{lem:effOM_effOF} and $\mu^t \in \PRnx{1}$.

	An alternative formulation of the mean-field equations is the mean-field PDE.
	First, we observe that \cref{eq:meso_X_flow} is precisely the characteristic ODE of a transport equation.
	Therefore, the mean-field PDE is 
	\newcommand{\crefMeanFieldPDE}{\cref{eq:meso_y_eff_pde,eq:meso_transport_pde}\xspace}
	\begin{subequations}
		\begin{empheq}[left={\empheqbiglbrace~}]{align}
			\label{eq:meso_y_eff_pde}
			\eff{m}(f,\Bfy) \, \ddot \Bfy 
			&= 
			\eff{F}(f,\Bfy,\dot \Bfy)
			\\
			\label{eq:meso_transport_pde}
			\partial_t f(x,t) &= -\mathrm{div}_x\Big(  f(x,t) \, \Phi(x,\Bfy(t))[ \dot \Bfy(t) ]\Big)
		\end{empheq}
	\end{subequations}
	where $f : \Rnx \times [0,T] \to [0,\infty)$ is the density of the particle position distribution $\mu^t$, i.e. $\dif \mu^t(x) = f(\dif x,t)$.
		
	\subsection{Consistency with the ODE model}
	
	We use empirical measures to show the relationship between the ODE model and the mean-field equations. For a particle position vector $\BfCalXN \in (\Rnx)^N$, we define the empirical measure
	\newcommand{\emp}{\mathrm{emp}}
	\begin{align}
		\mu_{\BfCalXN}^\emp \coloneqq \frac{1}{N} \sum_{i=1}^N \delta_{\BfX_i}
	\end{align}
	where $\delta_{\BfX_i}$ denotes the Dirac measure, which assigns mass $1$ to the point $\BfX_i$. The empirical measure is an exact representation of the particle positions $\BfCalXN$ up to permutations.
	
	By definition of the empirical measure, we obtain the relations
	\begin{align}
		\label{eq:emp_into_effm}
		\eff{m}(\mu^{\emp}_{\BfCalXN}, \Bfy) = \frac{1}{N} \sum_{i=1}^N \effO{m}(\BfXi,\Bfy) = \effN{m}(\BfCalXN,\Bfy)
	\end{align}
	and
	\begin{align}
		\label{eq:emp_into_effF}
		\eff{F}(\mu^{\emp}_{\BfCalXN}, \Bfy, \Bfv) = \frac{1}{N} \sum_{i=1}^N  \effO{F}(\BfXi,\Bfy, \Bfv) = \effN{F}(\BfCalXN,\Bfy, \Bfv).
	\end{align}	
	We see that the mean-field equations extend the ODE model, which is made precise in the following lemma. The idea is similar to \cite[Theorem 1.3.1]{GolseDynamicsLargeParticle2016}.

	\begin{lemma}[Consisitency with the ODE model]
		\label{lem:mean_field_consistency}
		Given a solution $(\BfCalXN, \Bfy)$ of the ODE model \crefODE with $\BfCalXN \in C^1([0,T],(\Rnx)^N)$ and $\Bfy \in C^2([0,T],\Rny)$, we define the empirical particle density $f^\pN$ by
		\[
		f^\pN(\dif x,t) \coloneqq \mu^\emp_{\BfCalXN(t)}.
		\] 
		Then, $(f^\pN, \Bfy)$ is a weak solution of the mean-field PDE \crefMeanFieldPDE with initial condition $\init \mu \coloneqq \mu^\emp_{\BfCalXN(0)}$. 
		Moreover, we can construct $X : [0,T] \times \Rnx \to \Rnx$ such that $(X, \Bfy)$ is a solution of the mean-field characteristic flow equations \crefCharFlowODE.
	\end{lemma}
	
	\begin{proof}
		
		We will first show how to construct a solution of \crefCharFlowODE. Let $X : [0,T] \times \Rnx \to \Rnx$
		be the solution of
	\begin{subequations}
	\begin{empheq}[left={\empheqbiglbrace~}]{align}
			\label{eq:tmp_X_flow}
		\dot X^t(x) &= \Phi( X^t(x),\Bfy(t))[ \dot \Bfy(t) ],\\
		X^0(x) &= x
		\end{empheq}
	\end{subequations}
		for all $x \in \Rnx$.
		The right-hand side is Lipschitz continuous in $X^t(x)$ by \cref{ass:L_phi} and continuous in $t$. Therefore, $X^t$ is well-defined by the Picard-Lindelöf Theorem.
		
		The flow $X^t$ is exactly the flow of the ODE \cref{eq:x_micro_ode}. If we evaluate $X^t$ at the initial particle positions $\init \BfCalX$, we obtain the original particle trajectories, i.e.
		\begin{align}
			\label{eq:flow_gives_partices_back}
			X^t( \init \BfXi ) = \BfXi(t) \quad \text{for all~} t \in [0,T].
		\end{align}
		
		Now, we define the initial particle measure as $\init \mu \coloneqq \mu^{\emp}_{\BfCalXN(0)}$
		and $\mu^t \coloneqq X^t \# \init \mu$ as solution candidate. Both measures are empirical measures, hence, $\init \mu, \mu^t \in \PRnx{1}$ for all $t \in [0,T]$.
		
		By \cref{eq:flow_gives_partices_back}, we obtain
		\begin{align}
			\label{eq:tmp_mu_t}
		\mu^t = X^t \# \init \mu = X^t \# \mu^{\emp}_{\BfCalXN(0)} = \mu^{\emp}_{\BfCalXN(t)}.
		\end{align}
		Using \cref{eq:y_micro_ode,eq:emp_into_effF,eq:tmp_mu_t}, we compute
		\begin{align*}
		\eff{m}(\mu^t,\Bfy) \ddot \Bfy 
		&= \effN{m}(\BfCalXN,\Bfy) \, \ddot \Bfy \\
		&= \effN{F}(\BfCalXN,\Bfy,\dot \Bfy) \\
		&= \eff{F}(\mu^t,\Bfy,\dot \Bfy).
		\end{align*}		
		As a result, $(\Bfy, X)$ is a solution of \crefCharFlowODE with initial conditions $\init \mu$.
		
		To show that $f^\pN$ is a weak solution of the mean-field PDE,
		we consider a test function with compact support $\xi \in C^\infty_c(\Rnx,\RR)$ and compute
		\begin{align*}
			\frac{\dif}{\dif t}
			\langle \, \xi , f^\pN \rangle
			&=
			\frac{\dif}{\dif t}
			\int \xi(x) \, \dif \mu^t(x)
			\\
			&=
			\frac{\dif}{\dif t}
			\int \xi(X^t(x')) \, \dif \init \mu(x')
			\\
			&=
			\int 
			\frac{\dif}{\dif t}
			\xi(X^t(x')) \, \dif \init \mu(x')
			\\
			&=
			\int 
			\partial_X \xi(X^t(x')) \, \dot X^t(x') \, \dif \init \mu(x')
			\\
			&=
			\int 
			\partial_X \xi(X^t(x')) \, \Phi(X^t(x'),\Bfy(t)) \, \dot \Bfy(t) \, \dif \init \mu(x')
			\\
			&=
			\int \partial_x \xi(x) \, \Phi(x,\Bfy(t))[\dot \Bfy(t)] f^\pN(x,t) \, \dif x
			\\
			&=
			- \int \xi(x) \, \mathrm{div}_x \left( \Phi(x,\Bfy(t))[\dot \Bfy(t)] f^\pN(x,t) \right) \, \dif x
			\\
			&= - \langle \xi \, , \mathrm{div}( (\Phi	 \, \dot \Bfy) f^\pN ) \rangle.
		\end{align*}
		Hence, $f^\pN$ is a weak solution of \cref{eq:meso_transport_pde}.	
	\end{proof}

	\subsection{Existence and uniqueness for the mean-field equations}
	
	The existence and uniqueness of solutions for the mean-field characteristics flow equations \crefCharFlowODE follows the same steps as in the discrete case. 
	
	As for the ODE model, we prove conservation of energy to justify the restriction to bounded velocities of the macroscopic component.
	
	\begin{lemma}[Conservation of energy]
		\label{lem:char_flow_conserves_energy}
		Let $(X, \Bfy, \dot \Bfy) \in C^1([0,T],Z)$ be a solution of \crefCharFlowODE.
		Then, the total energy
		\begin{align}
			E = \frac{1}{2} \Vert \dot \Bfy \Vert^2 + \mathcal{W}_0(\Bfy) + \int_\Rnx \left( \frac{m}{2}\Vert \Phi(x,\Bfy) \, \dot \Bfy \Vert^2 + \mathcal{W}_1(x) \right) \, \dif \mu^t(x) 
		\end{align}
		is conserved.
	\end{lemma}
	
	The proof is analogous to the proof of \cref{lem:ODE_conserves_energy}.
	
	\begin{proof}
		\label{proof:cons_energy_char_flow}
		
		For the following calculations, we use the shorthand notation
		\[
		\Phi_x \coloneqq \Phi(X^t(x),\Bfy)
		\quad \text{and} \quad
		\Omega_x \coloneqq \Omega(X^t(x),\Bfy).
		\]
		
		First, we take the time derivative of \cref{eq:meso_X_flow}, which yields
		\begin{align}
			\label{eq:tmp_meso_X_accel}
			\ddot X^t(x) &= \Phi(X^t(x),\Bfy(t))[ \ddot \Bfy(t)]
			+
			\Omega(X^t(x),\Bfy(t))[ \dot \Bfy(t), \dot \Bfy(t) ]
			\notag\\
			&= \Phi_x[ \ddot \Bfy] + \Omega_x[\dot\Bfy,\dot\Bfy].
		\end{align}

		Using \cref{eq:meso_y_eff,eq:meso_X_flow,eq:meso_mu_pushforward,eq:tmp_meso_X_accel}, we compute
		\begin{align*}
			\dot E
			&=
			\frac{\dif}{\dif t} \left(
			\frac{1}{2} \Vert \dot \Bfy \Vert^2 + \mathcal{W}_0(\Bfy) + \int_\Rnx \frac{m}{2}\Vert \Phi(x',\Bfy) \, \dot \Bfy \Vert^2 + \mathcal{W}_1(x') \, \dif \mu^t(x') 
			\right)
			\\
			&=
			\frac{\dif}{\dif t}
			\left(
			\frac{1}{2} \Vert \dot \Bfy \Vert^2 + \mathcal{W}_0(\Bfy) + \int_\Rnx \frac{m}{2}\Vert \Phi(X^t(x),\Bfy) \, \dot \Bfy \Vert^2 + \mathcal{W}_1(X^t(x)) \, \dif \init \mu(x) 
			\right)
			\\
			&=
			\dot \Bfy^T ( \ddot \Bfy - F_0(\Bfy) )
			+
			\frac{\dif}{\dif t}
			\int_\Rnx
			\frac{m}{2} \Vert \dot X^t(x) \Vert^2
			+ \mathcal{W}_1(X^t(x))
			\, \dif \init \mu(x)
			\\			
			&=
			\dot \Bfy^T ( \ddot \Bfy - F_0(\Bfy) )
			+
			\int_\Rnx
			(\dot X^t(x))^T ( m \ddot X^t(x) -  F_1(X^t(x))
			\, \dif \init \mu(x)
			\\
			&=
			\dot \Bfy^T ( \ddot \Bfy - F_0(\Bfy) )
			+
			\int_\Rnx
			\dot \Bfy^T \Phi_x^T ( m \Phi_x[ \ddot \Bfy] 
			+ m \Omega_x[ \dot \Bfy, \dot \Bfy] -  F_1(X^t(x))
			\, \dif \init \mu(x)
			\\
			&=
			\dot \Bfy^T \Big( 
			\Big(1+m \int_\Rnx \Phi_x^T \Phi_x \, \dif \init \mu(x) \Big)
			\ddot \Bfy
			\\
			& \quad \quad \quad - ( F_0(\Bfy) + \int_\Rnx \Phi_x^T \big(F_1(X^t(x)) - m\Omega_x[\dot \Bfy,\dot \Bfy]\big) \, \dif \init \mu(x) )
			\Big)
			\\
			&= \dot \Bfy^T ( \eff{m} \ddot \Bfy - \eff{F} )
			\\
			&= 0.
		\end{align*}
	\end{proof}
	
	To prove a Lipschitz bound for the mean-field characteristic flow equations \crefCharFlowODE, we first need to choose a function space for the flow $X^t$.
	We pick the function space of continuous and at most linearly growing functions, see \cite[Proof of Theorem 1.3.2]{GolseDynamicsLargeParticle2016}. We define
	\begin{align}
		Y \coloneqq \{ \varphi \mid \varphi \in C(\Rnx, \Rnx), \sup_{x \in \Rnx} \frac{\norm{\varphi(x)}}{1 + \norm{x}} < \infty \}
	\end{align}
	which forms a Banach space with the norm
	\[
	\norm{\varphi}_Y \coloneqq \sup_{x \in \RR^{n_x}}\frac{\norm{\varphi(x)}}{1 + \norm{x}}. 
	\]
	Our goal is to rewrite \crefCharFlowODE as an ODE of the form
	$
	\dot z = h(z).
	$
	Therefore, we define
	\begin{align}
	Z &\coloneqq Y \oplus \Rny \oplus \Rny.
	\end{align}
	The space $Z$ is a Banach space with the norm 
	\[
	\norm{(\varphi,\Bfy,\Bfv)}_Z = \norm{\varphi}_Y + \norm{\Bfy} + \norm{\Bfv}.
	\]
	Since conservation of energy implies that the velocity $\dot \Bfy$ will stay bounded for all times, we can restrict the domain of the right-hand side to the closed subset
	\begin{align}
		Z_{M_v} &\coloneqq \{ z = (\varphi, \Bfy, \Bfv) \in Z \mid \Vert \Bfv \Vert \leq M_v\} \subseteq Z
		\end{align}
	where $M_v>0$ is a constant.
	
	With the definition 
	\[
	h: Z_{M_v} \to Z: \begin{pmatrix}\varphi \\ \Bfy \\ \Bfv \end{pmatrix} \mapsto 
	\begin{pmatrix}
		x \mapsto \Phi(\varphi(x),\Bfy)[\Bfv]\\
		\Bfv \\
		(\eff{m}(\varphi \# \init \mu, \Bfy) )^{-1} \eff{F}(\varphi \# \init \mu, \Bfy, \Bfv)
	\end{pmatrix}
	\]
	we can rewrite \crefCharFlowODE as
	\[
	\dot z = h(z).
	\]
	Sometimes, we will write $h(\varphi,\Bfy,\dot\Bfy; \init \mu)$ to emphasise the dependency on $\init \mu$.
	The first component of the map $b$ is indeed an element of $Y$, since  \cref{ass:M_omega} implies\\ $\Vert \Phi(\varphi(\cdot),\Bfy)[\Bfv] \Vert_Y \leq M_\Phi M_v$ for all $(\varphi, \Bfy, \Bfv) \in Z_{M_v}$.

	\begin{theorem}[Existence and uniqueness] 
		\label{thm:pks_well_posed}
		For $\init \mu \in \mathcal{P}^1(\Rnx)$ and $\init \Bfy, \init \Bfv \in \Rny$, the system
		\crefCharFlowODE
		has a unique solution $(\Bfy,X) \in C^2([0,T],\Rny) \times C^1([0,T] \times \Rnx, \Rnx)$. 	
	\end{theorem}
		
	To proof existence and uniqueness, we will use the following Lemma.
	\begin{lemma}
		\label{lem:lipschitz_A}
		Let 
		$\phi : \RR^d \to \RR$ be Lipschitz continuous, $\varphi, \psi \in Y$ and $\mu \in \PRnx{1}$ then
		\[
		\left\vert \int_\Rnx \phi (\varphi(x)) \, \dif \mu(x) - \int_\Rnx \phi (\psi(x)) \, \dif \mu(x) \right\vert
		\leq
		\Lip(\phi ) C_{\mu} \Vert \varphi - \psi \Vert_Y
		\]
		where $C_{\mu} \coloneqq \int (1+\Vert x \Vert) \, \dif \mu(x) < \infty$.
	\end{lemma}
	
	\begin{proof}
		We compute
		\begin{align*}
			\MoveEqLeft[4]
			\left\vert \int_\Rnx \phi (\varphi(x)) \, \dif \mu(x) - \int_\Rnx \phi (\psi(x)) \, \dif \mu(x) \right\vert
			\\
			&\leq
			\Lip(\phi )
			\int_\Rnx \Vert \varphi(x) - \psi(x)\Vert \, \dif \mu(x)
			\\
			&=
			\Lip(\phi )
			\int_\Rnx \frac{\Vert \varphi(x) - \psi(x)\Vert}{1+\Vert x \Vert} (1+\Vert x \Vert) \, \dif \mu(x)
			\\
			&\leq \Lip(\phi ) C_\mu \Vert \varphi - \psi \Vert_Y.
		\end{align*}
	\end{proof}

\begin{proof}[Proof of \cref{thm:pks_well_posed}] 
	We define $M_v \coloneqq 2 E(\Bfy, \dot \Bfy, \init \mu)$.
	Then, \cref{lem:char_flow_conserves_energy} shows that every solution candidate has to satisfy $\Vert \dot \Bfy(t) \Vert \leq M_v$ for all $t \in [0,T]$.
	
	Now we show that the components of the map $h$ are Lipschitz continuous for all $(\varphi, \Bfy, \Bfv) \in Z_{M_v}$.
	
	We start with the component $\Phi(\varphi(\cdot),\Bfy)[\Bfv]$.
	The map  $(\varphi,\Bfy) \mapsto \Phi(\varphi(\cdot),\Bfy)$ is Lipschitz continuous by assumption \cref{ass:L_phi}, i.e. 
	\[
	\sup_{x \in \Rnx} \frac{\Vert \Phi(\varphi_1(x),\Bfy_1) - \Phi(\varphi_2(x),\Bfy_2) \Vert}{1 + \Vert x \Vert}
	\leq L_\Phi \left( \sup_{x \in \Rnx} \frac{\Vert \varphi_1(x) - \varphi_2(x) \Vert}{1+\Vert x \Vert} + \Vert \Bfy_1 - \Bfy_2 \Vert \right)
	\]
	and assumption \cref{ass:M_phi} yields the upper bound
	\[
	\Vert \Phi( \varphi(\cdot) , \Bfy ) \Vert_X \leq M_\Phi.
	\]
	Therefore, $(\varphi,\Bfy,\Bfv) \mapsto \Phi(\varphi(\cdot),\Bfy)[\Bfv]$ is Lipschitz for $\Vert \Bfv \Vert \leq M_v$.

	By \cref{lem:effOM_effOF}, we know that $\effO{m}$ and $\effO{F}$ are bounded and Lipschitz. Therefore, 
	\cref{lem:lipschitz_A} implies that 
	\[
	(\varphi,\Bfy,\Bfv) \mapsto \eff{m}(\varphi \# \init \mu,\Bfy,\Bfv)
	\]
	and
	\[
	(\varphi,\Bfy,\Bfv) \mapsto \eff{F}(\varphi \# \init \mu,\Bfy,\Bfv)
	\]
	are Lipschitz continuous and bounded for $\Vert \Bfv \Vert \leq M_v$.
	
	Moreover, $\eff{m}$ is symmetric and uniformly elliptic since integration of $\effO{m}$ preserves these properties as well.
	
	\cref{lem:m_inv_f} yields that for fixed $\init \mu$ the map 
	\[
	(\varphi,\Bfy,\Bfy) \mapsto h(\varphi,\Bfy,\dot\Bfy; \init \mu)
	\] is Lipschitz continuous provided $\Vert \Bfv \Vert \leq M_v$.
		
	Hence, the Picard-Lindelöf Theorem and \cref{lem:char_flow_conserves_energy} yield existence and uniqueness of solutions.
	
	We remark that the Lipschitz constant depends on $M_v$ and $C_{\init \mu}$.
	The latter is finite since $\init \mu \in \PRnx{1}$.
\end{proof}

	\subsection{Stability estimate}

   	The mean-field limit asserts that for $N \to \infty$ the behaviour of the ODE model \crefODE is well-approximated by the mean-field characteristic flow \crefCharFlowODE.
   	\cref{lem:mean_field_consistency} implies that solutions of the ODE model correspond to solutions of the mean-field equations \crefCharFlowODE with empirical measures as initial data. Hence, it is sufficient to show that solutions of \crefCharFlowODE depend continuously on the initial particle distribution $\init \mu$ to show convergence of the mean-field limit.
   	
   	As a metric between particle distributions,
   	we use the Monge-Kantorovich \\distance (also called Wasserstein distance).
	For a detailed study of the Monge-Kantorovich distance, we refer to \cite[Chapter 6]{VillaniOptimalTransport2009}.
	Since we will need the duality formula of the Monge-Kantorovich distance, we use this characterisation as a definition.

	\begin{definition}[Monge-Kantorovich distance {\cite[Proposition 1.4.2]{GolseDynamicsLargeParticle2016}}]
		\index{Monge-Kantorovich distance}
		\label{def:W1_dual}~\\
		For $\nu,\mu \in \PRnx{1}$, we define the \textit{Monge-Kantorovich distance (with exponent $1$)} as
		\begin{align}
			W_1(\nu,\mu) =
			\sup_{ \stackrel{\phi \in C(\Rnx,\Rnx)}{\Lip(\phi) \leq 1} }
			\left\vert \int_{\Rnx} \phi(x) \, \dif \nu(x) - \int_{\Rnx} \phi(x) \, \dif \mu(x) \right\vert 
			\label{eq:W1_dual}.
		\end{align}
	\end{definition}
	The Monge-Kantorovich distance is a complete metric on $\PRnx{1}$ \cite[Lemma 6.14]{VillaniOptimalTransport2009}.

	\begin{lemma}
		\label{lem:lipschitz_B}
		Let 
		$\phi : \RR^d \to \RR$ be Lipschitz continuous, then
		\[
		\left\vert \int_\Rnx \phi(x) \, \dif \mu_1(x) - \int_\Rnx \phi(x) \, \dif \mu_2(x) \right\vert
		\leq
		\Lip(\phi) W_1(\mu_1, \mu_2).
		\]
	\end{lemma}
	
	\begin{proof}
		This follows directly from the duality formula \cref{eq:W1_dual}.
	\end{proof}
			
	With the Monge-Kantorovich distance we can define the main result of this article.
	The following estimate is a generalisation of Dobrushin's stability estimate \cite[Section 3.3]{GolseDynamicsLargeParticle2016}.
			
	\begin{theorem}[Stability estimate]
		\label{thm:dobrushin}~\\
		\index{Stability!Dobrushin's estimate}\index{Dobrushin's stability estimate}
		Suppose that for $i \in \{1,2\}$ the tuples $(\Bfy_i(t),\mu^t_i(t)) \in \Rny \times \mathcal{P}^1(\Rnx)$ are solutions of  
		\crefCharFlowODE
		with initial conditions
		\[
		\Bfy_i(0) = \init \Bfy_i, \quad \dot \Bfy_i(0) = \init \Bfv_i \quad \text{and} \quad 
		\mu_i(0) = \init \mu_i.
		\]
		Then,
		\begin{align}
			\norm{\Bfy_1(t) - \Bfy_2(t)} + &\norm{\dot \Bfy_1(t) - \dot \Bfy_2(t)}
			+ W_1(\mu_1(t), \mu_2(t)) 
			\notag \\
			\MoveEqLeft \leq C e^{Lt} \left(
			\norm{\init \Bfy_1 - \init \Bfy_2} + \norm{\init \Bfv_1 - \init \Bfv_2} 
			+ W_1(\init \mu_1, \init \mu_2) \right),
			\label{eq:dobr_stab_estm}
		\end{align}
		for some constants $L, C > 0$ which do only depend on the initial total energy of both states and on the maximal first moment 
		\[
		C_\mu \coloneqq 
		\max
		\left(\int (1+\Vert x\Vert) \, \dif \init \mu_1(x), \int (1+\Vert x\Vert) \, \dif \init \mu_2(x)
		\right).
		\]
	\end{theorem}

	
	To prove this stability estimate, we use the following ODE estimate, see also \cite[Theorem 10.2]{HairerSolvingOrdinaryDifferential1993}.
	
	\begin{lemma}[{The Fundamental Lemma, {\cite[Chap 1., §7 Theorem VI]{WalterDifferentialIntegralInequalities1970}}}]
		\index{Stability!ODE}
		\label{lem:fundamental_ode}
		Let $A$ be a closed subset of a Banach space $Z$ and
		suppose that for $i \in \{1,2\}$ and $h_i : A \to Z$, the functions $z_i : [0,T] \to A$ solve
	\begin{subequations}
	\begin{empheq}[left={\empheqbiglbrace~}]{align}
			\dot z_i(t) &= h_i(z_i(t)) \quad \text{for~} t \geq 0,\\
			z_i(0) &= \init z_i.
		\end{empheq}
	\end{subequations}
		If for some constants $\varrho, \varepsilon, L > 0$ the following bounds hold
		\begin{enumerate}
			\item $\norm{\init z_1 - \init z_2} \leq \varrho$,
			\item $\norm{h_1(z) - h_2(z)} \leq \varepsilon \quad \text{for all~} z \in Z$,
			\item $\norm{h_1(z) - h_1(z')} \leq L \norm{z-z'} \quad \text{for all~} z,z' \in Z$,
		\end{enumerate}
		then,
		\begin{align}
			\norm{z_1(t) - z_2(t)} \leq \varrho e^{L t} + \frac{\varepsilon}{L} \left( e^{Lt} - 1 \right).
		\end{align}
	\end{lemma}

	\begin{proof}[Proof of \cref{thm:dobrushin}]
		We define 
		\begin{align*}
			h_i(\varphi,\Bfy,\Bfv) \coloneqq
			h(\varphi,\Bfy, \Bfv; \init{\mu}_i)
			\quad
			\text{for~} i \in \{1,2\}.
		\end{align*}
		
		For $i \in \{1,2\}$, we consider the ODEs 
	\begin{subequations}
	\begin{empheq}[left={\empheqbiglbrace~}]{align*}
			\dot z_i = h_i(z_i) \\
			z_i(0) = \init z_i
		\end{empheq}
	\end{subequations}
		with initial condition
		\[
		\init z_i \coloneqq ( x \mapsto x, \init \Bfy_i, \init{\Bfv}_i).
		\]
		
		Next, we check the conditions of \cref{lem:fundamental_ode}.\\
		1. We define
		\[
		\varrho \coloneqq \Vert \init z_1 - \init z_2 \Vert_Z
		=
		\Vert \init \Bfy_1 - \init \Bfy_2 \Vert
		+
		\Vert \init{\Bfv}_1 -  \init{\Bfv}_2 \Vert.
		\]		
		2. For arbitrary $(\varphi, \Bfy, \Bfv) \in Z_{M_v}$, we consider the maps
		\[
		\PRnx{1} \to \Rny : \mu \mapsto h(\varphi,\Bfy,\Bfv; \mu).
		\]
		We compute
		\begin{align*}
			\MoveEqLeft[2]\Vert h(\varphi,\Bfy,\Bfv; \mu_1) - h(\varphi,\Bfy,\Bfv; \mu_2) \Vert_Z			
			\\
			&=
			\Vert 
			(\eff{m}(\mu_1,\Bfy))^{-1} \eff{F}(\mu_1,\Bfy,\Bfv)
			- 
			(\eff{m}(\mu_2,\Bfy))^{-1} \eff{F}(\mu_2,\Bfy,\Bfv)\Vert.
		\end{align*}
		\cref{lem:lipschitz_B,lem:effOM_effOF} show that the maps
		\[
		\mu \mapsto \eff{m}(\mu,\Bfy)
		\quad \text{and} \quad
		\mu \mapsto \eff{F}(\mu,\Bfy,\Bfv)
		\]
		are bounded and Lipschitz continuous.
		Like in the proof of \cref{thm:pks_well_posed}, $\eff{m}(\mu,\Bfy)$ is also symmetric and uniformly elliptic for all $\mu \in \PRnx{1}$.
		Therefore, \cref{lem:m_inv_f} shows that $\mu \mapsto h(\varphi,\Bfy,\Bfv; \mu)$ is Lipschitz continuous, i.e.
		\[
		\Vert h_1(\varphi,\Bfy,\dot\Bfy) 
		- h_2(\varphi,\Bfy,\dot\Bfy) \Vert
		\leq
		L_\mu \, W_1(\init{\mu}_1, \init{\mu}_2) 
		\]
		and the Lipschitz constant $L_\mu$ depends only on $M_v$ and $C_\mu$.
		We set 
		\[
		\varepsilon \coloneqq L_\mu \, W_1(\init{\mu}_1, \init{\mu}_2).
		\]		
		3. From the proof of \cref{thm:pks_well_posed} we know that $h_1 : Z_{M_v} \to Z$ is Lipschitz continuous. The constant $\Lip(h_1)$ depends only on $M_v$ and $C_\mu$.
		\\
		
		Application of \cref{lem:fundamental_ode} yields
		\begin{align}
			\MoveEqLeft
			\Vert \Bfy_1(t) - \Bfy_2(t) \Vert
			+
			\Vert \dot \Bfy_1(t) - \dot \Bfy_2(t) \Vert
			+
			\Vert X^t_1 - X^t_2 \Vert_X
			\notag\\
			&\leq
			\varrho e^{t L_{f_1}} +  \frac{L_{\mu} W_1(\init \mu_1, \init \mu_2)}{\Lip(h_1)} ( e^{t \Lip(h_1)} - 1)		
			\label{eq:f1f2_stab_bound}
		\end{align}
		
		In comparison with the claim, equation \cref{eq:f1f2_stab_bound} does not include $W_1(\mu^t_1,\mu^t_2)$ yet. We derive one more estimate to fix this.
		For an arbitrary Lipschitz continuous map $\phi : \Rnx \to \RR$, we use \cref{lem:lipschitz_A} and \cref{lem:lipschitz_B} to compute
		\begin{align*}
			\MoveEqLeft[4] 
			\left\vert \int \phi(x) \, \dif \mu^t_1(x) - \int \phi(x) \, \dif \mu^t_2(x)
			\right\vert \\
			& \leq
			\left\Vert 
			\int \phi(X_1^t(x')) \, \dif \init \mu_1(x') - \int \phi(X_2^t(x')) \, \dif \init \mu_1(x') \right\Vert \\
			& \quad + 
			\left\Vert \int \phi(X_2^t(x')) \, \dif \init \mu_1(x') - \int \phi(X_2^t(x')) \, \dif \init \mu_2(x') \right\Vert
			\\
			&\leq \Lip(\phi) C_\mu \Vert X_1^t - X_2^t \Vert_X
			+ \Lip(\phi) W_1(\init \mu_1, \init \mu_2).
		\end{align*}
		
		Taking the supremum over all $\phi$ with $\Lip(\phi) \leq 1$ yields
		\begin{align}
			\label{eq:W1_bound_by_X}
			W_1(\mu^t_1, \mu^t_2)
			\leq
			C_\mu \Vert X^t_1- X^t_2\Vert_X
			+ W_1(\init \mu_1, \init \mu_2).
		\end{align} 
		
		Combining \cref{eq:f1f2_stab_bound} 
		and
		\cref{eq:W1_bound_by_X} yields
		the claim.

	\end{proof}
	
	\subsection{Convergence of the mean-field limit}
	
	With the stability estimate \cref{thm:dobrushin}, we can finally show convergence of the mean-field limit.

	\begin{corollary}
		\label{cor:mean-field-limit}
		We consider $\init \mu \in \PRnx{1}$, $\init \Bfy, \init \Bfv \in \Rny$ and a sequence of initial particle positions $\big( \init \BfCalX_k \big)_{k \in \mathbb N}$ such that
		\begin{align}
			W_1( \mu^{\emp}_{\init \BfCalX_k}, 
			\init \mu) \to 0 \quad \text{for~} k \to \infty.
			\label{eq:mean_field_limit_setup}
		\end{align}
		Let $(\mu^t, \Bfy(t))$ and $(\mu^t_k, \Bfy_k(t))$ be solutions for \crefCharFlowODE with initial conditions \\$(\init \mu, \init \Bfy, \init \Bfv)$ and   $( \mu^{\emp}_{\init \BfCalX_k}, \init \Bfy, \init \Bfv)$ for $k \in \mathbb N$.
		
		Then,
		\begin{align*}
			\norm{\Bfy(t) - \Bfy_k(t)} + \norm{\dot \Bfy(t) - \dot \Bfy_k(t)}
			+ 
			W_1(\mu^t, \mu^t_k) 
			\leq C e^{Lt} 
			W_1(\init \mu, \mu^{\emp}_{\init \BfCalX_k}) \to 0.
		\end{align*}
	\end{corollary}
	\cref{cor:mean-field-limit} follows directly from the stability estimate \cref{thm:dobrushin}.	
	\\
	
	We use \cref{cor:mean-field-limit} as the main mean-field limit result.
	For the subtleties of convergences of mean-field limits in different topologies, we refer to \cite[Section 1.4]{GolseDynamicsLargeParticle2016}.
	In particular, the existence of suitable initial conditions such that \cref{eq:mean_field_limit_setup} holds is shown in \cite[Section 1.4.4]{GolseDynamicsLargeParticle2016}.

	\begin{remark}[{Macroscopic equations}]
	Since the mean-field PDE \crefMeanFieldPDE is essentially a transport equation, there are several ways to obtain a macroscopic approximation for the particle dynamics. 
	One approach would be to assume that the density $f(x,t)$ has the shape of a normal distribution and then find equations for the first moments by inserting this ansatz into the \crefMeanFieldPDE.
	In the context of muscle models, this method is called distributed moment method \cite{ZahalakDistributionMomentApproximation1981}. The distributed moment method is the foundation for most numerical simulations of muscle tissue which include fibers, see for example \cite{BoelMicromechanicalModellingSkeletal2008, GfrererFiberBasedModeling2021, HeidlaufMultiScaleContinuum2016}.
	\end{remark}

	\section{Conclusion}
	We have generalised the results from our previous work \cite{PlunderCoupledSystemsLinear2020} to the case of nonlinear constraints. In particular, the same stability estimate as in the linear case also holds for nonlinear constraints. 
	Moreover, the proof follows Dobrushin's classical approach, which shows that many properties of unconstrained particle systems generalise well to particle systems with uniform, full-rank constraints.
	
	Nonetheless, our results are still restricted a particular class of particle systems with uniform, full-rank constraints. 
	We see two main directions for further study.
	First, with view towards possible applications, it would be desirable to extend our results to cases which cover more general constraints. For example, the particles could switch between a constrained case and an unconstrained case, as outlined in \cite[Section 5.3]{PlunderCoupledSystemsLinear2020}.
	Second, to further explore kinetic theory-related aspects, one could drop the full-rank condition and add interaction forces between the particles. This would lead to systems with similarities to Vlasov equations. In particular, if the constraints are not of full rank, the propagation of chaos becomes non-trivial since the state space of the particles could be both non-linear and dependent on the state of the macroscopic component.

	\section*{Acknowledgements} We want to thank Sara Merino-Aceituno for her numerous comments.
	Steffen Plunder is funded by the Vienna Science and Technology Fund (WWTF), grant VRG17-014.

	\appendix
	
	\section{Derivation of the ODE model}
	\label{sec:appendix}
	
	In this section, we show that solutions of the DAE model \crefDAE are also solutions of the ODE model \crefODE.
	We recall the DAE model	
	\begin{subequations}
		\begin{empheq}[left={\empheqbiglbrace~}]{align}
			\label{eq:app_x_micro_ind3}
			m \ddot{\BfX}_i = F_1(\BfX_i) - \nabla_{\BfX_i} g(\BfX_i,\Bfy) \Bf\lambda_i 
			& \quad \fai,
			\\
			\label{eq:app_y_micro_ind3}
			\ddot \Bfy = F_0(\Bfy) - \frac{1}{N} \sum_{j=1}^N \nabla_{\Bfy} g(\BfX_j,\Bfy) \Bf\lambda_j , &
			\\
			\label{eq:app_g_micro_ind3}
			g(\BfX_i,\Bfy) = g(\init \BfX_i, \init \Bfy)
			& \quad \fai
		\end{empheq}
	\end{subequations}
 	and the ODE model
	\begin{subequations}
 		\begin{empheq}[left={\empheqbiglbrace~}]{align}
 			\effN{m} \ddot \Bfy &= \effN{F}, 
 			\label{eq:app_y_micro_ode}\\
 			\dot \BfXi &= \Phi_i [ \dot \Bfy ], \fai.
 			\label{eq:app_x_micro_ode}
 		\end{empheq}
 	\end{subequations}
		
	\begin{proof}[Proof of \cref{prop:ode_dae}]
		In the following we assume $\BfXi \in C^1([0,T],\Rnx)$ and $\Bfy \in C^2([0,T],\Rny)$ for $1 \leq i \leq N$. We start with showing \enquote{$\cref{eq:app_g_micro_ind3} \Leftrightarrow \cref{eq:app_x_micro_ode}$}. Thereafter, we prove that if $(\BfCalXN,\Bfy)$ satisfy \cref{eq:app_x_micro_ode}, then \enquote{$\cref{eq:app_x_micro_ind3} \text{~and~} \cref{eq:app_y_micro_ind3} \Leftrightarrow \cref{eq:app_y_micro_ode}$}.
		
		1. We show that \cref{eq:app_g_micro_ind3} is equivalent to \cref{eq:app_x_micro_ode}.
		\begin{align}
			\cref{eq:app_g_micro_ind3}
			\Leftrightarrow 
			&~
			\frac{\mathrm{d} g(\BfXi,\Bfy)}{\mathrm{d} t} = \DifXi g (\BfXi,\Bfy) [\dot \BfX_i] + \Dify g (\BfXi,\Bfy) [\dot \Bfy] = 0
			\\
			\Leftrightarrow 
			&~ \dot \BfXi = 
			\underbrace{-\left( \DifXj g(\BfXi,\Bfy) \right)^{-1} 
			\Dify g(\BfXi,\Bfy)}_{= \Phi_i} [ \dot \Bfy ]
			\Leftrightarrow \cref{eq:app_x_micro_ode}
		\end{align}
		
		2. Under the condition that \cref{eq:app_x_micro_ode} holds, we show that
		\cref{eq:app_x_micro_ind3,eq:app_y_micro_ind3} are equivalent to \cref{eq:app_y_micro_ode}.
		
		Taking the time derivative of \cref{eq:app_x_micro_ode} implies
		\begin{align}
			\ddot \BfXi 
			&= \Phi_i[ \ddot \Bfy ] + \dot \Phi_i [ \dot \Bfy ] \notag \\
			&= \Phi(\BfXi,\Bfy)[ \ddot \Bfy ]+ \partial_{\BfXi} \Phi(\BfXi,\Bfy)[ \dot \Bfy, \dot \BfXi] + \partial_{\Bfy} \Phi(\BfXi,\Bfy)[\dot \Bfy, \dot \Bfy]
			\notag\\
			&= \Phi(\BfXi,\Bfy)[ \ddot \Bfy ]+ \partial_{\BfXi} \Phi(\BfXi,\Bfy)[ \dot \Bfy, \Phi(\BfXi,\Bfy)[ \dot \Bfy ]] + \partial_{\Bfy} \Phi(\BfXi,\Bfy)[\dot \Bfy, \dot \Bfy]
			\label{eq:tmp_g_micro_ind1_ddotX}
		\end{align}
	where we used \cref{eq:app_x_micro_ode} again in the last line to remove the dependency on $\dot \BfXi$.
	As a side-effect we obtain $\BfXi \in C^2([0,T],\Rnx)$.
	
	To shorten \cref{eq:tmp_g_micro_ind1_ddotX}, we define
	\begin{align}
		\Omega_i[\Bf v, \Bf w] 
		&\coloneqq \Omega(\BfXi, \Bfy)[\Bf v, \Bf w]
		\notag \\
		&\coloneqq \DifXi \Phi(\BfXi,\Bfy)[ \Bf v, \Phi(\BfXi,\Bfy)[ \Bf w ] ] + \Dify \Phi(\BfXi,\Bfy)[ \Bf v, \Bf w]
	\end{align}
	for $\Bf v, \Bf w \in \Rny$.
	Then, we can write \cref{eq:tmp_g_micro_ind1_ddotX} as
	\begin{align}
		\ddot \BfXi = \Phi_i[ \ddot \Bfy ] + \Omega_i [\dot \Bfy, \dot \Bfy].
		\label{eq:app_ddX}
	\end{align}
	With this preparation we can prove the desired equivalence.
	
	2.a) \enquote{$\cref{eq:app_x_micro_ind3} \text{~and~} \cref{eq:app_y_micro_ind3} \Rightarrow \cref{eq:app_y_micro_ode}$}
		
		We assume that $(\BfCalXN, \Bfy)$ are a solution of the DAE model. Then,
		\cref{eq:app_x_micro_ind3} and \cref{eq:app_ddX} imply
		\begin{align}
		\Bf\lambda_i
		&=
		(
		\nabla_{\BfXi} g(\BfXi,\Bfy)
		)^{-1}
		( 
		F_1(\BfX_i) 
		- m \ddot \BfX_i
		) 
		\label{eq:app_lambda_eq_dae} \\ 
		&=	(
		\nabla_{\BfXi} g(\BfXi,\Bfy)
		)^{-1}
		( 
		F_1(\BfX_i) 
		- m \Phi_i[ \ddot \Bfy ] - m\Omega_i [\dot \Bfy, \dot \Bfy]
		)
		\label{eq:app_lamda_formula}
		\end{align}
		Inserting \cref{eq:app_lamda_formula} into \cref{eq:app_y_micro_ind3} gives 
		\begin{align}
		\ddot \Bfy 
		= 
		F_0(\Bfy) 
		+ \frac{1}{N} \sum_{j=1}^N 
		\Phi_j^T
		(F_1(\BfXj) 
		- m 
		(	
		\Phi_j [ \ddot \Bfy ]
		+ \Omega_j [ \dot \Bfy , \dot \Bfy ]
		)
		).
		\label{eq:app_y_preode}
		\end{align}
		Collecting the acceleration terms to the left-hand side yields
		\begin{align}
			\underbrace{(I_{\Rny} + \frac{m}{N} \sum_{j=1}^N \Phi_j^T \Phi_j )}_{\eqqcolon \effN{m}}
			\,
			\ddot \Bfy 
			&= 
			\underbrace{
				F_0(\Bfy) 
				+ 
				\frac{1}{N} \sum_{j=1}^N \Phi_j^T  (   
				F_1(\BfXj) - m \Omega_i [ \dot \Bfy , \dot \Bfy ])
			}_{\eqqcolon \effN{F}}
		\end{align}
		which is equivalent to \cref{eq:app_y_micro_ode}.

		2.b) \enquote{$\cref{eq:app_x_micro_ind3} \text{~and~} \cref{eq:app_y_micro_ind3} \Leftarrow \cref{eq:app_y_micro_ode}$}
		We assume now that $(\BfCalXN, \Bfy)$ solves the ODE model.
		
		To find a solution of the DAE model, we define
		\begin{align}
		\Bf \lambda_i \coloneqq	(
		\nabla_{\BfXi} g(\BfXi,\Bfy)
		)^{-1}
		( 
		F_1(\BfX_i) 
		- m \Phi_i[ \ddot \Bfy ] - m\Omega_i [\dot \Bfy, \dot \Bfy]
		)
		\label{eq:app_lambda_def}
		\end{align}
		which is exactly \cref{eq:app_lamda_formula}.
		The assumptions $g \in C^2$, $\nabla_\BfXi g(\BfXi,\Bfy)$ invertable and $\Bfy \in C^2$ imply  $\Bf\lambda_i \in C^0([0,T],\Rnx)$.
		
		Inserting \cref{eq:app_ddX} into \cref{eq:app_lambda_def} yields \cref{eq:app_lambda_eq_dae} which is equivalent to \cref{eq:app_x_micro_ind3}.
		
		And as a final step, inserting \cref{eq:app_lambda_def} into  \cref{eq:app_y_micro_ode} yields \cref{eq:app_y_micro_ind3}.
	\end{proof}

	\bibliographystyle{apalike}  
	\bibliography{plunder_main}

	\medskip
	\medskip
	
\end{document}